\documentclass[preprint,3p,authoryear]{elsarticle}
\journal{Frontiers of Mathematics in China}
\usepackage{amsthm,amssymb,amsmath}
\usepackage{enumitem}
\usepackage{algorithm}
\usepackage{algorithmic}
\usepackage{multirow}
\usepackage{xcolor}
\usepackage{booktabs} 
\usepackage{graphicx}
\usepackage{float}
\usepackage{hyperref}

\usepackage{url}
\newtheorem{theorem}{Theorem}[section]
\newtheorem{definition}{Definition}[section]

\newtheorem{corollary}{Corollary}[section]
\newtheorem{lemma}{Lemma}[section]

\newcommand{\E}{\mathbb{E}}

\begin{document}

\begin{frontmatter}
\title{The COM-negative binomial distribution: modeling overdispersion and ultrahigh zero-inflated count data}

\address[label1]{School of Mathematical Sciences and Center for Statistical Science, Peking University, Beijing, 100871, China}
\address[label2]{School of Statistics, East China Normal University, Shanghai, 200241, China}
\address[label3]{School of Mathematics and Statistics, Central China Normal University, Wuhan, 430079, China}

\author[label1]{Huiming Zhang}
\ead{zhanghuiming@pku.edu.cn}

\author[label2]{Kai Tan}
\ead{kaitan@stu.ecnu.edu.cn}

\author[label3]{Bo Li}
\ead{haoyoulibo@163.com}

\begin{abstract} In this paper, we focus on the COM-type negative binomial distribution with three parameters, which belongs to COM-type $(a,b,0)$ class distributions and family of equilibrium distributions of arbitrary birth-death process. Besides, we show abundant distributional properties such as overdispersion and underdispersion, log-concavity, log-convexity (infinite divisibility), pseudo compound Poisson, stochastic ordering and asymptotic approximation. Some characterizations including sum of equicorrelated geometrically distributed random variables, conditional distribution, limit distribution of COM-negative hypergeometric distribution, and Stein's identity are given for theoretical properties. COM-negative binomial distribution was applied to overdispersion and ultrahigh zero-inflated data sets. With the aid of ratio regression, we employ maximum likelihood method to estimate the parameters and the goodness-of-fit are evaluated by the discrete Kolmogorov-Smirnov test.
\end{abstract}
\end{frontmatter}

\section{Introduction}

Before 2005, Conway-Maxwell-Poisson distribution (denoted as COM-Poisson distribution) had been rarely used since \cite{conway62} briefly introduced it for modeling of queuing systems with state-dependent service time, see also \cite{wimmer99}, \cite{wimmer95}. About ten years ago, the COM-Poisson distribution with two parameters was revived by \cite{shmueli05} as a generalization of Poisson distribution. More recently, there has been a fast growth of researches on COM-Poisson distribution in terms of related statistical theory and applied methodology, see \cite{sellers12} and the references therein. The probability mass function~(p.m.f.) of the COM-Poisson random variable (r. v) $X$ is given by
\begin{equation} \label{eq:com}
\mathrm{P}(X = k) = \frac{{{\lambda ^k}}}{{{{(k!)}^\nu }}} \cdot \frac{1}{{Z(\lambda ,\nu )}},(k = 0,1,2, \cdots ),
\end{equation}
where $\lambda ,\nu  > 0$ and $Z(\lambda ,\nu ) = \sum\limits_{i = 0}^\infty  {\frac{{{\lambda ^i}}}{{{{(i!)}^\nu }}}} $. We denote (\ref{eq:com}) as $X \sim {\rm{CMP}}(\lambda ,\nu )$.

\cite{kokonendji08} proved that COM-Poisson distribution was overdispersed when $\nu  \in {\rm{[0,1)}}$ and underdispersed when $\nu  \in {\rm{(1, + }}\infty {\rm{)}}$. Another extension of Poisson is negative binomial, which is a noted discrete distribution with overdispersion property and is widely applied in actuarial sciences(see \cite{denuit07}, \cite{kaas08}). The p.m.f. of the negative binomial r.v $X$ is
\begin{equation} \label{eq:nbd}
\mathrm{P}(X = k) =\frac{{\Gamma (r + k)}}{{k!{\mkern 1mu} \Gamma (r)}}{p^k}{(1 - p)^r}, (k = 0,1,2, \ldots),
\end{equation}
where $r \in (0,\infty )$ and $p \in (0,1)$.

 In this paper, we propose a COM-negative binomial (denoted by CMNB) distribution , which extends the negative binomial distribution and depends on three parameters by replacing $\frac{{\Gamma (r + k)}}{{k!{\mkern 1mu} \Gamma (r)}}$ in (\ref{eq:nbd}) with $({\frac{{\Gamma (r + k)}}{{k!{\mkern 1mu} \Gamma (r)}})^\nu }$ and divide the normalization constant $C(r,\nu ,p) = \sum\limits_{i = 0}^\infty  {(\frac{{\Gamma (r + i)}}{{i!{\mkern 1mu} \Gamma (r)}})}^{\nu} {p^i}{(1 - p)^r}$.

\begin{definition}
A r.v. $X$ is said to follow CMNB distribution $({\rm{CMNB}}(r,\nu ,p))$ with three parameters $(r,\nu ,p)$ if the p.m.f. is given by
\begin{equation} \label{eq:cnbd}
\mathrm{P}(X = k) =
\frac {{{{(\frac{{\Gamma (r + k)}}{{k!{\mkern 1mu} \Gamma (r)}})}^\nu }{p^k}{{(1 - p)}^r}}}
      {{\sum\limits_{i = 0}^\infty  {{{(\frac{{\Gamma (r + i)}}{{i!{\mkern 1mu} \Gamma (r)}})}^\nu }} {p^i}{{(1 - p)}^r}}}
={{\left(\frac{{\Gamma (r + k)}}{{k!{\mkern 1mu} \Gamma (r)}}\right)}^\nu }
{{p^k}{{(1 - p)}^r}}
\frac{1}{{C(r,\nu ,p)}},\quad (k = 0,1,2, \ldots ),
\end{equation}
where $r,\nu \in (0,\infty )$ and $p \in (0,1)$.
\end{definition}

In another point of view, the alternative form of (\ref{eq:cnbd}) can be written as
\begin{equation} \label{eq:cnbda}
{\rm{P}}(X = k) = \frac{{{{[\frac{{\Gamma (r + k)}}{{k!\Gamma (r)}}{{\tilde p}^k}{{(1 - {\tilde p})}^r}]}^\nu }}}{{\sum\limits_{i = 0}^\infty  {[{{\frac{{\Gamma (r + i)}}{{i!\Gamma (r)}}}^\nu }} {{\tilde p}^i}{{(1 - {\tilde p})}^r}{]^\nu }}} = {\left( {\frac{{\Gamma (r + k)}}{{k!\Gamma (r)}}} \right)^\nu }{\tilde p^k}{(1 - \tilde p)^r}\frac{1}{{C(r,\nu ,\tilde p)}},\quad (k = 0,1,2, \ldots ),
\end{equation}
where $\tilde p = {p^{1/\nu }}$.

When $r \le 1$, we will show that CMNB (\ref{eq:cnbd}) is discrete compound Poisson, which has wide application in risk theory (includes non-life insurance) as well, see \cite{zhang14} and the  references therein. For illustrating the p.m.f of $X$ defined in (3), we plot 12 cases of CMNB in Figture \ref{cmnb_plot}.

It is easy to see that our CMNB distribution belongs to the COM-type extension of $(a,b,0)$ class. The $(a,b,0)$ class distribution is a famous family of distributions which sometimes refers to Katz class (see remarks in section 2.3.1 of \cite{johnson05}). It has significant applications in non-life insurance mathematics, especially for modelling claim counts~( loss models, collective risk models), see \cite{denuit07}. A classic result in non-life insurance textbooks states that the $(a,b,0)$ class distribution only contains degenerate, binomial, Poisson and the negative binomial distribution. After adding a new parameter $\nu  \in {{\rm{R}}^ + }$, we define the COM-type extension of $(a,b,0)$ distribution, and it is convenient to see that degenerate, COM-Poisson, COM-binomial and CMNB belong to this class of distributions.

\cite{shmueli05} firstly proposed the COM-binomial distribution which is presented as a sum of equicorrelated Bernoulli variables. \cite{borges14} studied some properties and an asymptotic approximation (e.g. COM-binomial approximates to COM-Poisson under some conditions) of this family of distributions in detail. We will show that some results of COM-Poisson can be extended in our CMNB distribution.  \cite{kadane15} gave the exchangeably properties, sufficient statistics and multivariate extension of COM-binomial distribution. Another variant of CMNB distribution has been studied by \cite{imoto14}, it just replaced the term ${\Gamma (r + k)}$ in (\ref{eq:nbd}) by ${\Gamma (r + k)}^\nu $ and then divided the normalization constant. We will give adequate reasons to support our extension in succeeding sections.  \cite{chakraborty16} considered the extended COM-Poisson distribution (ECOMP$(r,\theta ,\alpha ,\beta )$):
\begin{equation} \label{eq:E}
\mathrm{P}(X = k) = \frac{{\Gamma {{(r + k)}^\beta }}}{{{{(k!)}^{\alpha }}}}{\theta ^k}/\sum\limits_{i = 1}^\infty  {\frac{{\Gamma {{(r + i)}^\beta }}}{{{{(i!)}^{\alpha  }}}}{\theta ^i}} \quad (k = 0,1,2, \ldots ),
\end{equation}
where the parameter space is $(r \ge 0,\theta  > 0,\alpha  > \beta ) \cup (r > 0,0 < \theta  < 1,\alpha  = \beta )$. ECOMP distribution combines \cite{imoto14}'s extension ($\alpha  = 1 $), \cite{chakrabortyOng16}'s extension (COMNB$(r,\theta ,\alpha ,\beta )$ with $\beta = 1 $) and the CMNB distribution (\ref{eq:cnbd}) with $\alpha   = \beta  = \nu $. The COM-Poisson is a special case of ECOMP when $\beta  = 0 $. ECOMP distribution (\ref{eq:E}) has the queuing systems characterization (birth-death process with arrival rate ${{\lambda _k} = {{(r + k)}^\nu }}$ and service rate ${{\mu _k} = {k^\nu }\mu }$ for $k \ge 1$ ), see also \cite{brown01} for arbitrary birth-death process characterization.

The rest of the article is organized as the follows. In section 2, we propose the COM-type $(a,b,0)$ class distributions and demonstrate some example of COM-type $(a,b,0)$ class which includes the CMNB. Further more, some properties of CMNB, such as Renyi entropy and Tsallis entropy representation, overdispersion and underdispersion, log-concavity, log-convexity (infinite divisibility), pseudo compound Poisson, stochastic ordering and asymptotic approximation are studied. In section 3, some conditional distribution characterizations and Stein identity characterization are presented by using related lemmas, and we also show that COM-negative hypergeometric can approximate to CMNB. In section 4, inverse methods were introduced to generate CMNB distributed random variables. Section 5 then estimates the parameters by maximum likelihood method, in which the initial values are provided by recursive formula. In section 6, two simulated data sets and two applications to actuarial claim data sets are given as examples. In section 7, we provide some potential and further research suggestions based on the properties and characterizations of CMNB distribution. 

\begin{figure}[!ht]
\centering
\includegraphics[scale=0.8]{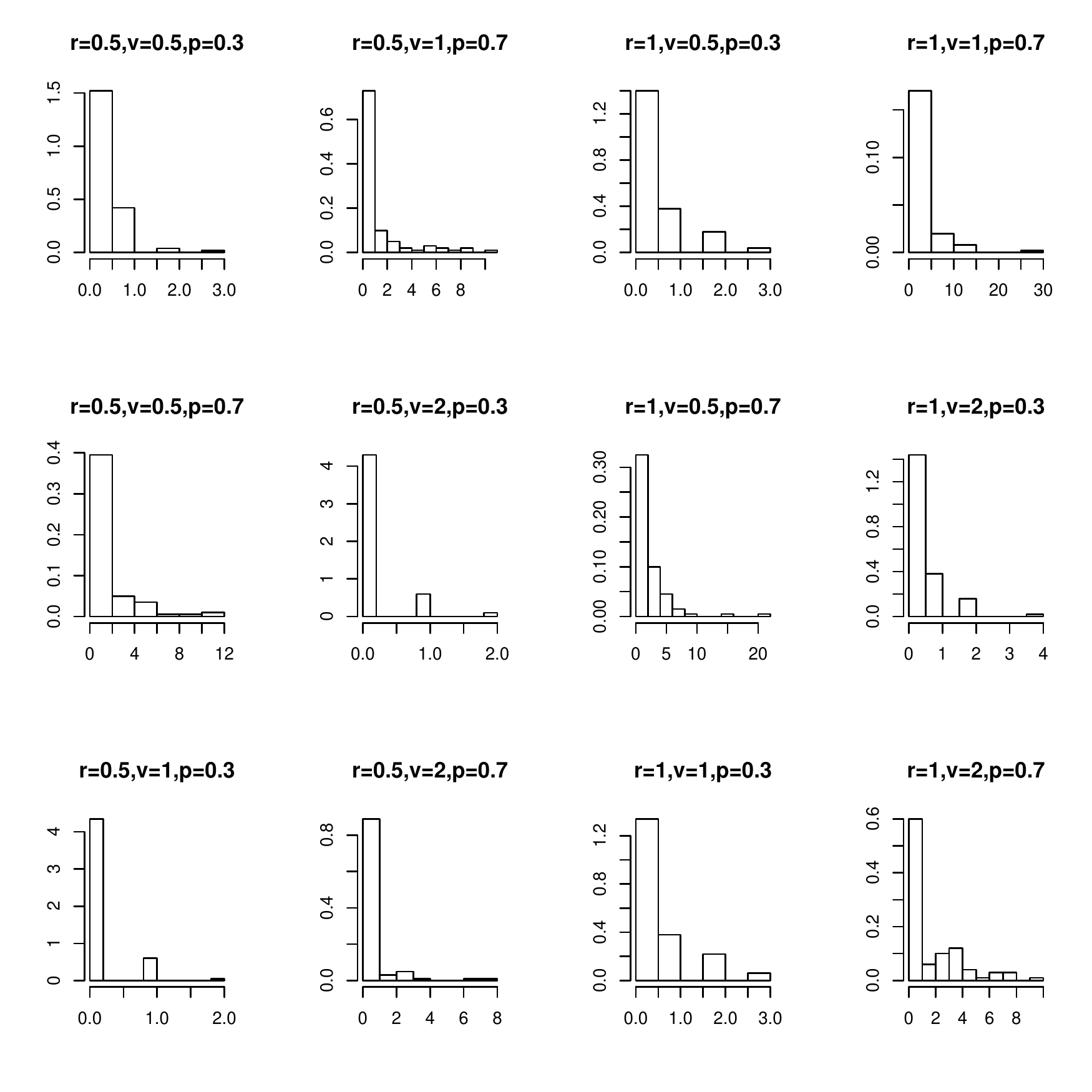}
\caption{Some plots of p.m.f. (\ref{eq:cnbd}) for  $r \in \{ 0.5,1\}$ ,$\nu  \in \{ 0.5, 1, 2\} $  and $p \in \{ 0.3,0.7\} $.}\label{cmnb_plot}
\end{figure}

\section{Properties}

\subsection{Recursive formula and ultrahigh zero-inflated property}

The recursive formula (or ratios of consecutive probabilities) is given by
\begin{equation} \label{eq:rf}
\frac{{\mathrm{P}(X = k)}}{{\mathrm{P}(X = k - 1)}} = p{\left( {\frac{{\Gamma (k + r)}}{{k!\Gamma (r)}}} \right)^\nu }{\raise0.7ex\hbox{${}$} \!\mathord{\left/
 {\vphantom {{} {}}}\right.\kern-\nulldelimiterspace}
\!\lower0.7ex\hbox{${}$}}{\left( {\frac{{\Gamma (k - 1 + r)}}{{(k - 1)!\Gamma (r)}}} \right)^\nu } = p \cdot {\left( {\frac{{k - 1 + r}}{k}} \right)^\nu }.
\end{equation}

We say that a zero-inflated count data $X$ following some discrete distribution is ultrahigh zero-inflated if $\mathrm{P}(X = 0)/\mathrm{P}(X = 1) \gg 1 $. For CMNB case with two parameters, we have
$\frac{\text{P}_{\text{CMNB}}(X=0)}{\text{P}_{\text{CMNB}}(X=1)}= \frac 1p(\frac 1r)^{\nu}$;
and for negative binomial case, we get
$\frac{\text{P}_{\text{NB}}(X=0)}{\text{P}_{\text{NB}}(X=1)}= \frac{1}{{pr}}.$
If we choose ${r < 1}$ and ${\nu > 1}$ in CMNB case, then this CMNB distribution is more flexible to deal with the ratio $\frac{{\mathrm{P}(X = 0)}}{{\mathrm{P}(X = 1)}}$ comparing to the NB distribution, as $\frac{1}{p}{\left( {\frac{1}{r}} \right)^\nu } \gg \frac{1}{{pr}}$ when $\nu >1$. (For examples, the plots of CMNB distribution with ${\nu = 2}$ in Figure \ref{cmnb_plot}; $\nu =10.4$ in Table \ref{real_2} of section 6.2). In the insurance company, the more zero insurance claims, the less risk to bankrupt. The following definition provides a generalization of $(a,b,0)$ class distribution.

\begin{definition}
Let $X$ be a discrete r.v., if ${p_k} = \mathrm{P}(X = k)$ satisfies the recursive formula
\begin{equation}
{p_k} = {(a + \frac{b}{k})^\nu }{p_{k - 1}},{\rm{   (}}k = 1,2, \cdots )
\end{equation}
for some constants $a,b \in {\rm{R}}$ and $\nu  \in {{\rm{R}}^ + }$, then we call it COM-type $(a,b,0)$ class distribution. We denote this class as $COM(a,b,\nu ,0)$ .
\end{definition}

The COM-Poisson distribution ${\rm{CMP}}(\lambda ,\nu )$ satisfies the case $a=0$, since ${p_k} = \frac{\lambda }{{{k^\nu }}}{p_{k - 1}}$ and $b={\lambda ^{1/\nu }}$.

From (\ref{eq:rf}), it is easy to see that CMNB distribution belongs to COM-type $(a,b,0)$ class distribution with ${p_0} = 1/C(r,\nu ,p)$.
\begin{equation}
\frac{{{p_k}}}{{{p_{k - 1}}}} = p{\left( {1 + \frac{{r - 1}}{k}} \right)^\nu } = {\left( {{p^{1/\nu }} + \frac{{(r - 1){p^{1/\nu }}}}{k}} \right)^\nu }.
\end{equation}

The COM-binomial distribution (CMB), see \cite{shmueli05}, \cite{borges14}, with p.m.f.
\begin{equation} \label{eq:cb}
\mathrm{P}(X = k) = \frac{{{{{{m} \choose {k}}}^v}{p^k}{{(1 - p)}^{m - k}}}}{{\sum\limits_{i = 0}^m {{{{{m} \choose {i}}}^v}{p^i}{{(1 - p)}^{m - i}}} }},\quad k = 0,1, \cdots ,m,
\end{equation}
where $\nu  \in {{\rm{R}}^ + }, m \in {{\rm{Z}}^ + }, p \in (0,1)$. We denote (\ref{eq:cb}) as $X \sim {\rm{CMB}}(m, p,\nu )$. Since the ratio of consecutive probabilities is $\frac{{\mathrm{P}(X = k)}}{{\mathrm{P}(X = k - 1)}} = \frac{p}{{1 - p}}{\left( {\frac{{m + 1 - k}}{k}} \right)^\nu }$, COM-binomial distribution belongs to COM type $(a,b,0)$ class distribution.\\[2mm]
\textbf{Remark 1}: As we know, the $(a,b,0)$ class distribution only contains degenerate distribution, binomial distribution, Poisson distribution and the negative binomial distribution. But there are other distributions belongs to COM type $(a,b,0)$ class. For example, $m \in {{\rm{Z}}^ + }$ can be replaced by $m \in {{\rm{R}}^ + }$ in (\ref{eq:cb}), and the p.m.f. is given by
$$
\mathrm{P}(X = k) = \frac{{{{{{{m} \choose {k}}}^{2}}}{p^k}{{(1 - p)}^{m - k}}}}{{\sum\limits_{i = 0}^{\infty } {{{{{m} \choose {i}}}^{2}}{p^i}{{(1 - p)}^{m - i}}} }},\quad k = 0,1,2, \cdots ,
$$
where $p/(1 - p) < 1$ such that $\sum\limits_{i = 0}^\infty  {{{{m} \choose {i}}^2}{p^i}{{(1 - p)}^{m - i}}}  < \infty $.\\[2mm]
\textbf{Remark 2}: \cite{brown01} considered the very large class of stationary distribution of birth-death process with arrival rate ${\lambda_k}$ and service rate ${\mu_k}$ by the recursive formula:
\[\frac{{\mathrm{P}(X = k)}}{{\mathrm{P}(X = k - 1)}} = \frac{{{\lambda _{k - 1}}}}{{{\mu _k}}},\quad (k =1,2
, \ldots ).\]

Thus we can construct a birth-death process with arrival rate ${\lambda_k}=c{[a(k+1) + b]^\nu }$ and service rate ${\mu_k}=c{k^\nu }$, and $c$ is a positive constant, which characterizes the COM-type $(a,b,0)$ class distribution.
\cite{}
\subsection{Related to R{\'e}nyi entropy and Tsallis entropy}

Notice the R{\'e}nyi entropy (see \cite{renyi61}) in the information theory, which generalizes the Shannon entropy. The Renyi entropy of order $\alpha$ of a discrete r.v. $X$:
$$
H_\alpha ^R(X) = \frac{1}{{1 - \alpha }}{\rm{ln}}\sum\limits_{i = 0}^\infty  {[P} (X = i){]^\alpha },(\alpha  \ne 1).
$$

Let $X$ be negative binomial distributed in (\ref{eq:nbd}) and $X_\nu$ be CMNB distributed in (\ref{eq:cnbda}). Then the normalization constant $C(r,\nu ,\tilde p)$ in (\ref{eq:cnbda}) has R{\'e}nyi entropy representation $C(r,\nu ,\tilde p) = {e^{(1 - \nu )H_\nu ^R(X)}}$,
so $P({X_\nu } = x) = {P^\nu }(X = x)/{e^{(1 - \nu )H_\nu ^R(X)}}$.

Another generalization of Shannon entropy in physic is the Tsallis entropy. For r.v. $X$, its Tsallis entropy of order $\alpha$ is defined by
$$
H_\alpha ^T(X) = \frac{1}{{1 - \alpha }}\left( {\sum\limits_{i = 0}^\infty  {[P} (X = i){]^\alpha } - 1} \right),(\alpha  \ne 1).
$$
This entropy was introduced by \cite{tsallis88} as a basis for generalizing the Boltzmann-Gibbs statistics. Also, the normalization constant $C(r,\nu ,\tilde p)$ in (\ref{eq:cnbda}) has Tsallis entropy representation $C(r,\nu ,\tilde p) = 1 + (1 - \nu )H_\nu ^T(X)$, then $P({X_\nu } = x) = {P^\nu }(X = x)/[1 + (1 - \nu )H_\nu ^T(X)]$.

\subsection{Log-concave, Log-convex, Infinite divisibility}

This subsection deals with log-concavity and log-convexity of the CMNB distribution. A discrete distribution with ${p_k} = \mathrm{P}(X = k)$ is said to have log-concave (log-convex) p.m.f. if
\[\frac{{{p_{k + 1}}{p_{k - 1}}}}{{p_k^2}} = \frac{{{p_{k + 1}}}}{{{p_k}}}/\frac{{{p_k}}}{{{p_{k - 1}}}} \le ( \ge )1, \quad k \ge 1.\]

\begin{lemma}
The CMNB distribution is log-concave if $r  \ge 1$ and log-convex if $r \le 1$.
\end{lemma}
\begin{proof}
In fact, using the ratio of consecutive probabilities (\ref  {eq:rf}), we have
$$
M=\frac{{{p_{k + 1}}}}{{{p_k}}}/\frac{{{p_k}}}{{{p_{k - 1}}}}=p{\left( {\frac{{r + k}}{{k + 1}}} \right)^\nu }/p{\left( {\frac{{r + k - 1}}{k}} \right)^\nu } = {\left( {\frac{{{k^2} + kr}}{{{k^2} + kr + r - 1}}} \right)^\nu }.
$$
Then $M \le 1$ iff $r \ge 1$(log-concave) and $M \ge 1$ iff $r \le 1$(log-convex).
\end{proof}
{\noindent\textbf{Remark 3}:} \cite{ibragimov08} called a distribution strongly unimodal if it is unimodal and its convolution with any unimodal distribution is unimodal. He showed that the strongly unimodal distributions is equivalent to the log-concave distributions. So CMNB distribution has strong unimodality(see Figure \ref{cmnb_plot} for example) when $r \ge 1$. \cite{steutel70} showed that all log-convex discrete distributions are infinitely divisible, the background and detailed proof can be found in \cite{steutel03}. Then we obtain infinite divisibility of CMNB distribution when $r \le 1$.

\begin{corollary}\label{cor:id}
The CMNB distribution (\ref{eq:cnbd}) is discrete infinitely divisible (discrete compound Poisson distribution) if $r \le 1$.
\end{corollary}

Feller's characterization of the discrete infinite divisibility showed that a non-negative integer valued r.v. $X$ is infinitely divisibleiff its distribution is a discrete compound Poisson distribution with probability generating function (in short, p.g.f.):
\begin{equation}\label{eq:dcp}
G(z) = \sum\limits_{k = 0}^\infty  {{p_k}{z^k}}  = {e^{\sum\limits_{i = 1}^\infty  {{\alpha _i}\lambda ({z^i} - 1)} }}{\rm{,(}}\left| z \right| \le {\rm{1) }},
\end{equation}
where $\sum _{i = 1}^\infty {\alpha _i} = 1,{\alpha _i} \ge 0,\lambda  > 0.$

For a theoretical treatment of discrete infinite divisibility (or discrete compound Poisson distribution), we refer readers to section 2 of \cite{steutel03}, section 9.3 of \cite{johnson05}, \cite{zhang16}.

Considering some ${\alpha _i}$ being negative in (\ref{eq:dcp}), it turns into a generalization of the discrete compound Poisson distribution:
\begin{definition}[Discrete pseudo compound Poisson distribution] \label{def:DPCP}
  If a discrete r.v.~$X$ with $ \mathrm{P}(X = k) = p_{k}$, $k\in \Bbb N$, has a
  p.g.f.~of the form
  \begin{equation}
   G(z) = \sum\limits_{k = 0}^\infty  p_{k} z^{k} = \exp \left\{\sum\limits_{i= 1}^{\infty} \alpha_{i} \lambda (z^{i} - 1) \right\},
  \end{equation}
  where $\sum\limits_{i = 1}^{\infty} \alpha_{i} = 1$, $\sum \limits_{i =
    1}^{\infty} \left| \alpha_{i} \right| < \infty$, $\alpha_{i} \in
  \Bbb{R}$, and $\lambda > 0$, then $X$ is said to follow a discrete pseudo compound Poisson distribution, abbreviated as \emph{DPCP}.
\end{definition}

Next, we will give two lemmas on the non-vanishing p.g.f. characterization of DPCP, see \cite{zhang14} and \cite{zhang17}.

\begin{lemma} \label{lem:char-DPCP}
 Let $p_k=\mathrm{P}(X=k)$, for any discrete r.v.~$X$, its p.g.f.~$G(z) = \sum\limits_{k = 0}^\infty  {{p_k}} {z^k}$ has no zeros in $ - 1 \le z \le 1$iff $X$ is DPCP distributed.
\end{lemma}

The proof of Lemma \ref{lem:char-DPCP} is based on Wiener-L{\'e}vy theorem, which is a sophisticated theorem in Fourier analysis, see \cite{zygmund02}.

\begin{lemma}(L{\'e}vy-Wiener theorem)
Let $F(\theta ) = \sum\limits_{k =  - \infty }^\infty  {{c_k}{e^{ik\theta }}} ,\theta  \in [0,2\pi ]$ be a absolutely convergent Fourier series with $\left\| F \right\| = \sum\limits_{k =  - \infty }^\infty  {\left| {{c_k}} \right|}  < \infty $. The value of $F(\theta )$ lies on a curve $C$, and $H(t)$ is an analytic (not necessarily single-valued) function of a complex variable which is regular at every point of $C$. Then $H[F(\theta )]$ has an absolutely convergent Fourier series.
\end{lemma}

\begin{lemma}\label{lem:mnt}
For any discrete r.v. $X$ with p.g.f. $G(z) = \sum\limits_{k = 0}^\infty  {{p_k}} {z^k},(\left| z \right| \le 1)$. If ${p_0} > {p_1} > {p_2} >  \cdots$ , then $X$ is DPCP distributed.
\end{lemma}
\begin{proof}
First, we show that $G(z)$ has no zeros in $\left| z \right| < 1$, since
\begin{align*}
\left| {(1 - z)G(z)} \right| &=
\left| {{p_0} - ({p_0} - {p_1})z - ({p_1} - {p_2}){z^2} +  \cdots } \right| \\
&\ge {p_0} - \left| {({p_0} - {p_1}) |z| + ({p_1} - {p_2})| {{z^2}} | +  \cdots } \right|\\
&> {p_0} - \left| {({p_0} - {p_1})  + ({p_1} - {p_2}) +  \cdots } \right| =p_0 - |p_0|=0.
\end{align*}
And notice that $G(1) = 1,G( - 1) = {p_0} - {p_1} + {p_2} - {p_3} +  \cdots  > 0$, so $z =  \pm 1$ are not zeros point.
\end{proof}
The condition in the next corollary is weaker than that of Corollary \ref{cor:id}, and the result (DPCP) is also weaker than Corollary \ref{cor:id} (DCP).

\begin{corollary}
The CMNB distribution (\ref{eq:cnbd}) is discrete pseudo compound Poisson distribution if $(p{r^\nu } < 1, r > 1)$ or $(r \le 1)$.
\end{corollary}

\begin{proof}
 On the one hand, $r \le 1$ deduces that CMNB belongs to discrete compound Poisson by Corollary \ref{cor:id},
 hence CMNB is discrete pseudo compound Poisson. On the other hand, by using \textbf{Lemma \ref{lem:mnt}}, we need to guarantee that $\frac{{\mathrm{P}(X = k)}}{{\mathrm{P}(X = k - 1)}} = p{(\frac{{k - 1 + r}}{k})^\nu } < 1$ for $k = 1,2, \cdots $. The $p(\frac{{k - 1 + r}}{k})^\nu $ is a decreasing function with respect to $k$ when $r > 1$, $\frac{{\mathrm{P}(X = k)}}{{\mathrm{P}(X = k - 1)}}$ reaches its maximum $p{r^\nu } $ as $k=1$. So $p{r^\nu } < 1$ is the other case.
\end{proof}

Applying the recurrence relation (L{\'e}vy-Adelson-Panjer recursion) of p.m.f. of DPCP distribution, see Remark 1 in \cite{zhang14}
 \[{P_{n + 1}} = \frac{\lambda }{{n + 1}}[{\alpha _1}{P_n} + 2{\alpha _2}{P_{n - 1}} +  \cdots  + (n + 1){\alpha _{n + 1}}{P_0}],\quad ({P_0} = {e^{ - \lambda }},n = 0,1, \cdots )\]
 and ${P_k} = \mathrm{P}(X = k) = {(\frac{{\Gamma (r + k)}}{{k!\Gamma (r)}})^\nu }{p^k}{P_0}$, then the DPCP parametrization $(\lambda ,{\alpha _1},{\alpha _2}, \cdots )$ of CMNB is determined by the following system of equations:
\[{(\frac{{\Gamma (r + n + 1)}}{{(n + 1)!}})^\nu }{p^{n + 1}} = \frac{\lambda }{{n + 1}}[{\alpha _1}{(\frac{{\Gamma (r + n)}}{{n!}})^\nu }{p^n} + 2{\alpha _2}{(\frac{{\Gamma (r + n - 1)}}{{(n - 1)!}})^\nu }{p^{n - 1}} +  \cdots  + (n + 1){\alpha _{n + 1}}],(n = 0,1, \cdots ),\]
where $\lambda  = \log {P_0}$.

\subsection{Overdispersion and underdispersion}
In statistics, for a given random sample $X$, overdispersion means that $\mathrm{E}[X]<\mathrm{Var}(X)$. Conversely, underdispersion means that $\mathrm{E}[X]>\mathrm{Var}(X)$. Moreover, equal-dispersion means that $\mathrm{E}[X]=\mathrm{Var}(X)$. \cite{gomez11} summerized the phenomena of insurance count claims data, which were characterized by two features: (i) Overdispersion, i.e., the variance is greater than the mean; (ii) Zero-inflated, i.e. the presence of a high percentage of zero values in the empirical distribution.

The CMNB distribution belongs to the family of weighted Poisson distribution (see \cite{kokonendji08}) with p.m.f.
\begin{equation} \label{eq:wpd}
\mathrm{P}(X = k) = \frac{{w(k)}}{{{\rm{E}}[w(X)]}} \cdot \frac{{{\theta ^k}}}{{k!}}{e^{ - \theta}},
\end{equation}
where $w(k)$ is a non-negative weighted function.

Then weighted Poisson representation of CMNB distribution is
\[\mathrm{P}(X = k) = \frac{{{{(1 - p)}^r}{e^p}}}{{C(r,\nu ,p)}}{[\Gamma (1 + k)]^{1 - \nu }}{[\Gamma (r + k)]^\nu }\frac{{{p^k}}}{{k!}}{e^{ - p}},\quad (k = 0,1,2, \ldots ).\]

Therefore, CMNB distribution in (\ref{eq:cnbd}) can be seen as a weighted Poisson distribution with weighted function
\begin{equation} \label{eq:weight}
f(k,r,\nu ) = w(k) = {[\Gamma (1 + k)]^{1 - \nu }}{[\Gamma (r + k)]^\nu }.
\end{equation}

Theorem 3 and its corollary in \cite{kokonendji08} provide an ``iff" condition to prove overdispersion and underdispersion of the weighted Poisson distribution.

\begin{lemma}\label{lem:wf}
 Let $X$ be a weighted Poisson random variable with mean $\theta >0$, and let $w(k),k \in {\rm{N}}$ be a weighted function not depending on $\theta$. Then, weighted function $k \mapsto w(k)$ is logconvex (logconcave) iff the weighted version $X_w$ of $X$ is overdispersed (underdispersed).
\end{lemma}

\cite{kokonendji08} applied it to show that COM-Poisson distribution is overdispersion if $\nu  < 1$ and underdispersion if $\nu  > 1$. We employ their methods to get a criterion for overdispersion or underdispersion of CMNB distribution.

\begin{theorem}\label{thm:wpd}
Set ${\Delta _k} = \sum\limits_{i = 0}^\infty  {(\frac{{1 - \nu }}{{{{(i + k + 1)}^2}}} + \frac{\nu }{{{{(i + k + r)}^2}}})} $.
The COM-nagative binomial distribution (\ref{eq:cnbd}) is overdispersion if ${\Delta _k} > 0 $ $(\forall k \in {\rm{N)}}$ and  underdispersion if ${\Delta _k} < 0$ $(\forall k \in {\rm{N)}}$.
\end{theorem}
\begin{proof}
Function $f(x)$ is logconvex~(logconcave) if $\frac{{{d^2}\log f(x)}}{{d{x^2}}}> 0( < 0)$. Followed by the formula of logarithmic second derivative of Gamma function~(see p54 of \cite{temme11}), $\frac{{{d^2}\log \Gamma (x)}}{{d{x^2}}} = \sum\limits_{i = 0}^\infty  {\frac{1}{{{{(x + i)}^2}}}}$,  we have
\[\frac{{{{d}^2}\log f(k,r,\nu )}}{{{d}{k^2}}} = \frac{{(1 - \nu ){{d}^2}\log \Gamma (k + 1)}}{{{d}{k^2}}} + \frac{{\nu {{d}^2}\log \Gamma (k + r)}}{{{d}{k^2}}} = \sum\limits_{i = 0}^\infty  {(\frac{{1 - \nu }}{{{{(i + k + 1)}^2}}} + \frac{\nu }{{{{(i + k + r)}^2}}})}, (\forall k \in {\rm{N}}).\]
Applying Lemma \ref{lem:wf}, the proof is complete.
\end{proof}

Then, the results of overdispersion can be easily obtained by Theorem \ref{thm:wpd}.

\begin{corollary}
In these two cases: 1. $\nu  > 0,r < 1$; 2. $\nu  < 1,r \in {{\rm{R}}^ + }$. CMNB distribution is overdispersion.
\end{corollary}
{\noindent \textbf{Remark 4}}: The result of case 2 ($\nu  < 1,r \in {{\rm{R}}^ + }$) can be also obtained from Corollary \ref{cor:id} and overdispersion of discrete compound Poisson distribution (equivalently, the discrete infinitely divisible).

\subsection{Stochastic ordering}

Stochastic ordering is the concept of one r.v. $X$ neither stochastically greater than, less than nor equal to another r.v. $Y$. There are plenty types of stochastic orders, which have various applications in risk theory. Firstly, we present 4 different definitions for discrete r.v.: usual stochastic order, likelihood ratio order, hazard rate order and mean residual life order.

1. $X$ is stochastically less than $Y$ in usual stochastic order (denoted by $X{ \le _{st}}Y$) if ${S_X}(n) \ge {S_Y}(n)$ for all $n$, where ${S_X}(n) = \mathrm{P}(X \ge n) = \sum\limits_{k = n}^\infty  {{p_k}} $ is the survival function $X$ of with p.m.f. $p_k$ .

2. $X$ is stochastically less than $Y$ in likelihood ratio order (denoted by $X{ \le _{lr}}Y$) if $\frac{{g(n)}}{{f(n)}}$ increases in $n$ over the union of the supports of $X$ and $Y$, where $f(n)$ and $g(n)$ denotes the p.m.f. of $X$ and $Y$, respectively.

3. $X$ is stochastically less than $Y$ in hazard rate order (denoted by $X{ \le _{hr}}Y$) if ${r_X}(n) \ge {r_Y}(n)$ for all $n$, where the hazard function of a discrete r.v. $X$ with p.m.f. $p_k$ is defined as $r_X(n) = {p_n}/\sum\limits_{k= n}^\infty  {{p_k}}$.

4. $X$ is stochastically less than $Y$ in mean residual life order~(denoted by $X{ \le _{MLR}}Y$) if ${\mu _X}(n) \ge {\mu _Y}(n)$ for all $n$, where the mean residual life function of a discrete r.v. $X$ with p.m.f. $p_k$ is defined as ${\mu _X}(n) = {\rm{E}}(X - n\left| {X \ge n} \right.) = \sum\limits_{k= n}^\infty  {k{p_k}} /\sum\limits_{k= n}^\infty  {{p_k}}  - n$.

The relationship among the above four stochastic ordering are $X{ \le _{lr}}Y \Rightarrow X{ \le _{hr}}Y \Rightarrow X{ \le _{MLR}}Y$\\
(see Theorem 1.C.1 of \cite{shaked07}) and $X{ \le _{hr}}Y \Rightarrow X{ \le _{st}}Y$ (see Theorem 1.B.1 of \cite{shaked07}).

\cite{gupta14} gave the stochastic ordering between COM-Poisson r.v. $X$ and Poisson distributed r.v. $Y$ with same parameter $\lambda $ in (\ref{eq:com}), that is $X{ \le _{lr}}Y$,  therefore $X{ \le _{st}}Y$, $X{ \le _{hr}}Y$ and $X{ \le _{MLR}}Y$. In the following result we will show that CMNB distribution also has some stochastic ordering properties.
\begin{theorem}\label{thm:so}
Let $X$ and $Y$ be two r.vs following CMNB distribution with parameters $(r,{\nu}_1 ,p)$ and $(r,{\nu}_2 ,p)$, respectively. If ${\nu _1} \le {\nu _2},r \ge 1$, then $X{ \le _{lr}}Y$, hence $X{ \le _{st}}Y$, $X{ \le _{hr}}Y$ and $X{ \le _{MLR}}Y$.
\end{theorem}
\begin{proof}
Note that $r \ge 1$, we have $\frac{{\Gamma (r + n + 1)}}{{(n + 1)!}}/\frac{{\Gamma (r + n)}}{{n!}} = \frac{{r + n}}{{n + 1}} \le 1$. Then
\[\frac{{{\rm{P}}(Y = n)}}{{{\rm{P}}(X = n)}} = {\left(\frac{{\Gamma (r + n)}}{{n!\Gamma (r)}}\right)^{{\nu _2} - {\nu _1}}}\frac{{C(r,{\nu _1},p)}}{{C(r,{\nu _2},p)}},\quad (n = 0,1,2, \ldots ).\]
which is increasing in $n$ as ${\nu _1} \ge {\nu _2}$.
\end{proof}

Especially, assume that $r$ is a positive integer, the CMNB should be called the COM-Pascal distribution. Let $X$ be CMNB distributed and $Y$ be negative binomial distributed with the same parameters $r,p$,  it yields to $X{ \le _{lr}}Y$ when ${\nu _2} > {\nu _1} = 1$.

The next theorem is proved in the view of weighted Poisson distribution (\ref{eq:wpd}) from weighted function of CMNB distribution.  Example 1.C.59 of \cite{shaked07} states the obvious lemma below:

\begin{lemma}
Define $X_w$ as the r.v. with weighted density function ${f_w}(x) = \frac{{w(x)}}{{{\rm{E[}}w(X)]}}f(x),(x \ge 0)$,
Similarly, for another nonnegative r.v. $Y$ with density function $g$, define $Y_w$ as the r.v. with the weighted density function ${g_w}(y) = \frac{{w(y)}}{{{\rm{E[}}w(Y)]}}g(y),(y \ge 0)$. If $w(x)$ is an increasing function, then $X{ \le _{hr}}Y \Rightarrow {X_w}{ \le _{hr}}{Y_w}$.
\end{lemma}

\begin{theorem}
Let $X$ and $Y$ be two CMNB distributed with parameters $(r,{\nu} ,p_1)$ and $(r,{\nu} ,p_2)$, respectively. If $({p_1} \le {p_2}, \nu  \le 1)$ or $({p_1} \le {p_2}, r \ge 1)$, then $X{ \le _{lr}}Y$, and therefore $X{ \le _{st}}Y$, $X{ \le _{hr}}Y$ and $X{ \le _{MLR}}Y.$
\end{theorem}
\begin{proof}
For Poisson distributed $X$,$Y$ with mean $p_1$,$p_2$, if ${p_1} \le {p_2}$, then ${\rm{P}}(Y = n)/{\rm{P}}(X = n) = {(\frac{{{p_2}}}{{{p_1}}})^n}{e^{ - ({p_2} - {p_1})}}$
is increasing for all $n$. So $X{ \le _{hr}}Y$. From section 2.4, we know that CMNB is weight Poisson with weight  (\ref{eq:weight}).

 On the one hand, when $\nu  \le 1$, we notice that weighted density function $w(x) = {[\Gamma (1 + x)]^{1 - \nu }}{[\Gamma (r + x)]^\nu }$ for CMNB distribution is increasing with respect to $x$. On the other hand, $w(x) = \Gamma (1 + x){\left( {\frac{{\Gamma (r + x)}}{{\Gamma (1 + x)}}} \right)^\nu }$ is an increasing function with respect to $x$ as $r \ge 1$, that is, ${X_w}{ \le _{hr}}{Y_w}$.
\end{proof}

\subsection{Approximate to COM-Poisson distribution}
The next theorem enables CMNB distribution to be a suitable generalization since its limit distribution is the COM-Poisson under some conditions. We prove that CMNB distribution converges to the COM-Poisson distribution when $r$ goes to infinity.

\begin{theorem}\label{thm:foc}
Suppose that r.v. $X$ has CMNB distribution with parameters $(r,\nu ,p)$, denote the p.m.f. as $\mathrm{P}(\left. {X = k} \right|r,\nu ,p)$, and let $\lambda  = {r^\nu}\frac{p}{{1 - p}}$. Then
\begin{equation}
\lim_{r\to\infty}\mathrm{P}(X = k\mid r,\nu ,p)=\frac{{{\lambda ^k}}}{{{{(k!)}^\nu }}} \cdot \frac{1}{{Z(\lambda ,\nu )}},(k = 0,1,2, \cdots ).
\end{equation}
\end{theorem}
\begin{proof}
Notice that $p = \frac{\lambda }{{{r^\nu} + \lambda }}$, substitute to p.m.f (\ref{eq:cnbd}), then we obtain
\[\mathrm{P}(X = k\mid r,\nu ,p)
= \frac{{{\lambda ^k}}}{{{{(k!)}^\nu }}} \cdot
{\left( {\frac{{\Gamma (r + k)}}{{\Gamma (r)\;{r^k}}}} \right)^\nu }
\cdot \frac{1}{{{{(1 + \lambda /{r^\nu})}^k}}}
\left/ \frac{{C(r,\nu ,p)}}{{( {\frac{{{r^\nu}}}{{{r^\nu} + \lambda }}} )^r}}
\right..
\]
Hence,
\[\lim_{r\to\infty}\mathrm{P}(X = k\mid r,\nu ,p) =
\frac{{{\lambda ^k}}}{{{{(k!)}^\nu }}}
\left/ \frac{{C(r,\nu ,p)}}{{( {\frac{{{r^\nu}}}{{{r^\nu} + \lambda }}} )^r}}
\right.
= \frac{{{\lambda ^k}}}{{{{(k!)}^\nu }}} \cdot \frac{1}{{Z(\lambda ,\nu )}}\]
holds as $\mathop {\lim }\limits_{r \to  + \infty } {\left( {\frac{{\Gamma (r + k)}}{{\Gamma (r)\;{r^k}}}} \right)^\nu } \cdot \frac{1}{{{{(1 + \lambda /{r^v})}^k}}} = 1,(k = 0,1,2, \cdots )$ and
\[\mathop {\lim }\limits_{r \to  + \infty } \frac{{C(r,\nu ,p)}}{{( {\frac{{{r^v}}}{{{r^v} + \lambda }}} )^r}}
= \mathop {\lim }\limits_{n \to  + \infty } \mathop {\lim }\limits_{r \to  + \infty } \sum\limits_{i = 0}^n {\frac{{{\lambda ^i}}}{{{{(i!)}^\nu }}} \cdot {{\left( {\frac{{\Gamma (r + i)}}{{\Gamma (r)\;{r^i}}}} \right)}^\nu } \cdot \frac{1}{{{{(1 + \lambda /{r^\nu})}^i}}}}  = \mathop {\lim }\limits_{n \to  + \infty } \sum\limits_{i = 0}^n {\frac{{{\lambda ^i}}}{{{{(i!)}^\nu }}}}  = Z(\lambda ,\nu ).\]
\end{proof}

\section{Characterizations}
\subsection{Sum of equicorrelated geometrically distributed r.v.}
It is well known that the binomial r.v. can be seen as the sum of $m$ independent Bernoulli r.v. $Z_i$.
\[S = {Z_1} + {Z_2} +\cdots  + {Z_m}\]
where
\[{\mathrm{P}}({Z_i} = 1) = p,\,{\mathrm{P}}({Z_i} = 0) = 1 - p,\,i = 1,2, \cdots ,m\]
\[\mathrm{P}(S = k) = {{m} \choose {k}} {p^k}{(1 - p)^{m - k}}.\]

\cite{shmueli05} and \cite{borges14} mentioned that the COM-binomial distribution (\ref{eq:cb}) can be presented as a sum of equicorrelated Bernoulli r.vs $\{ {Z_i}\} _{i = 1}^m$ with joint distribution
\[\mathrm{P}({Z_1} = {z_1},\cdots,{Z_m} = {z_m}) = \frac{{{ {{m} \choose {k}}^{v - 1}}{p^k}{{(1 - p)}^{m - k}}}}{{\sum\limits_{{x_1} = 0}^1 \cdots{\sum\limits_{{x_m} = 0}^1 {{{{m} \choose {\sum_{i=1}^{m}x_i}}^{v - 1}}} {p^{\sum_{i=1}^{m}x_i}}{{(1 - p)}^{m - \sum_{i=1}^{m}x_i}}} }}, z = ({z_1},\cdots,{z_m}) \in {\{ 0,1\} ^m},\]
where $k=\sum_{i=1}^{m}z_i$.

As we know, negative binomial distribution can be treated as the sum of $m$ independent geometric r.vs $Z_i$
$(i=1,\cdots,m)$:
\[S = {Z_1} + {Z_2} +\cdots  + {Z_m}\]
\[\mathrm{P}(S = x) = {{m+x-1} \choose {x}}{p^{ x}}{(1 - p)^m},\]
where
$\mathrm{P}({Z_i} = {z_i}) =  p^{z_i}(1 - p), (z_i=1,2,,\cdots ). $

It is similar to see that the CMNB distribution can be interpreted as a sum of equicorrelated geometric r.vs $Z_i$
$(i=1,\cdots,m)$ with joint distribution
\begin{equation} \label{eq:gs1}
\mathrm{P}({Z_1} = {z_1}, \cdots ,{Z_m} = {z_m}) \propto
{{m+x-1} \choose {x}}^{v - 1}{p^{x}}{(1 - p)^m},
\end{equation}
where $x=\sum_{i=1}^{m}z_i$.

The reason is that we assume $\{ {Z_i}\} _{i = 1}^m$ is equicorrelated, and $x=\sum_{i=1}^{m}z_i$ has ${{m+x-1} \choose {x}}$ feasible positive integer solutions, and each solution has probability $\mathrm{P}({Z_1} = {z_1}, \cdots ,{Z_m} = {z_m})$. Then, the ${{m+x-1} \choose {x}}$ possible values of random vector $({Z_1}, \cdots ,{Z_m})$ such that $S=\sum_{i=1}^{m}Z_i$ is CMNB distributed, namely
$$
{{m+x-1} \choose {x}} \mathrm{P}({Z_1} = {z_1}, \cdots ,{Z_m} = {z_m}) \propto
{{m+x-1} \choose {x}}^{v }{p^{x}}{(1 - p)^m}.
$$
Thus we have (\ref{eq:gs1}).

\subsection{Conditional distribution}
In this subsection, two conditional distribution characterizations are obtained for CMNB distribution. For two independent r.vs $X,Y$, what is the form of the conditional distribution of $X$ given $S=X+Y$? Consider the sum of CMNB r.vs with parameters $(r_x,\nu ,p)$ and $(r_y,\nu ,p)$, then
\begin{align*}
\mathrm{P}(S = s) &= \sum\limits_{x = 0}^s {\mathrm{P}(X = x)} \mathrm{P}(Y = s - x)
= \sum\limits_{x = 0}^s {{{\left(\frac{{\Gamma ({r_x} + x)}}{{x!\Gamma ({r_x})}}\right)}^\nu }
                                 \frac{{{p^x}{{(1 - p)}^{{r_x}}}}}{{C({r_x},\nu ,p)}}}
                                 {\left(\frac{{\Gamma ({r_y} + s - x)}}{{(s - x)!\Gamma ({r_y})}}\right)^\nu }
                                 \frac{{{p^{s - x}}{{(1 - p)}^{{r_y}}}}}{{C({r_y},\nu ,p)}}\\
 &= \frac{{{{(1 - p)}^{{r_x}}}{{(1 - p)}^{{r_y}}}}}{{C({r_x},\nu ,p)C({r_y},\nu ,p)}}\sum\limits_{x = 0}^s {{{\left( {\frac{{\Gamma ({r_x} + x)\Gamma ({r_y} + s - x)}}{{x!\Gamma ({r_x})(s-x)!\Gamma ({r_y})}}} \right)}^\nu }} {p^s}\\
 &= \frac{{{{(1 - p)}^{{r_x}}}{{(1 - p)}^{{r_y}}}{{[\Gamma ({r_x} + {r_y} + s)]}^\nu }}}{{C({r_x},\nu ,p)C({r_y},\nu ,p){{[\Gamma ({r_x} + {r_y})]}^\nu }}}\sum\limits_{x = 0}^s {{{\left( {{s} \choose {x}}{\frac{{{\rm{B}}({r_x} + x,{r_y} + s - x)}}{{{\rm{B}}({r_x},{r_y})}}} \right)}^\nu }} {p^s}.
 \end{align*}
The conditional distribution $\mathrm{P}(X = k\left| {S = s} \right.)$ is
\begin{equation} \label{eq:cnh}
 \frac{{\mathrm{P}(X = k)\mathrm{P}(Y = s - k)}}{{\mathrm{P}(S = s)}} = {\left( {{{s} \choose {k}}\frac{{{\rm{B}}({r_x} + k,{r_y} + s - k)}}{{{\rm{B}}({r_x},{r_y})}}} \right)^\nu }\left /\sum\limits_{x = 0}^s {{{\left( {{{s} \choose {x}} \frac{{{\rm{B}}({r_x} + x,{r_y} + s - x)}}{{{\rm{B}}({r_x},{r_y})}}} \right)}^\nu }}.\right.
\end{equation}
Using (\ref{eq:cnh}), we naturally define the p.m.f. of COM-negative hypergeometric distribution with parameter $(z,\nu,m,n)$ as follow:
\begin{equation} \label{eq:cnhv}
\mathrm{P}(X = k)= {\left( {{{z} \choose {k}}\frac{{{\rm{B}}({m} + k,{n} + z - k)}}{{{\rm{B}}({m},{n})}}} \right)^\nu }/\sum\limits_{x = 0}^z {{{\left( {{{z} \choose {x}} \frac{{{\rm{B}}({m} + x,{m} + z - x)}}{{{\rm{B}}({m},{n})}}} \right)}^\nu }} = \frac{{\left( {{{z} \choose {k}}\frac{{{\rm{B}}({m} + k,{n} + z - k)}}{{{\rm{B}}({m},{n})}}} \right)^\nu }}{{N(m,n,z,\nu )}},
\end{equation}
where $k = 0,1,2, \cdots ,z$ and $N(m,n,z,\nu )$ is the normalization constant.

When $\nu=1$, COM-negative hypergeometric distribution turns out to be negative hypergeometric distribution, see \cite{wimmer99}, \cite{johnson05}.

 By \cite{patil64}'s general characterization theorem for negative binomial, Poisson and geometric distribution, we know that, given $X + Y= x+y$, if the conditional distribution
$ X \left| {X+Y } \right.$ is negative hypergeometric with parameters $m$ and $n$ for all values of the sum $x + y$,
\[\mathrm{P}(X=x \left| {X + Y=x+y} \right.) = {{x+y} \choose {x}}\frac{{{\rm{B}}(m + x,n + y)}}{{{\rm{B}}(m,n)}},\]
then $X$ and $Y$ both are negative binomial distribution, with parameters $(m,p)$ and $(n,p)$,
$$f(x) =\frac{{\Gamma (m + x)}}{{x!{\mkern 1mu} \Gamma (m)}}{p^x}{(1 - p)^m},\ g(y) =\frac{{\Gamma (n + y)}}{{y!{\mkern 1mu} \Gamma (n)}}{p^y}{(1 - p)^n},$$
respectively (see also \cite{kagan73}).

\begin{lemma}\label{eq:patil}(\cite{patil64})
Let $X$ and $Y$ be independent and both discrete (or both continuous) r.vs and suppose $\mathrm{P}(X \left| {X+Y } \right.)$ is the function $c(x,x + y)$. If $\frac{{c(x + y,x + y)c(0,y)}}{{c(x,x + y)c(y,y)}}$ is of the form $h(x + y)/h(x)h(y)$  where $h( \cdot )$ is an arbitrary non-negative function, then
\begin{equation} \label{eq:fg}
f(x) = f(0)h(x){e^{ax}},g(y) = g(0)\frac{{h(y)c(0,y)}}{{c(y,y)}}{e^{ay}},
\end{equation}
where $\mathrm{P}(X=x)=f(x) > 0, \mathrm{P}(Y=y)=g(y) > 0$ and $f(0),g(0)$ are the corresponding normalizer for $f(x)$ and $g(y)$ which make them p.m.f..
\end{lemma}
Now we apply Lemma \ref{eq:patil} for characterizing CMNB distribution.
\begin{theorem}
Let X,Y be the independent discrete r.v. with $\mathrm{P}(X=x)=f(x)>0 $ and $\mathrm{P}(Y=y)=g(y)>0$.
If the $\mathrm{P}(X=x \left| {X+Y=x+y } \right.)$ is the COM-negative hypergeometric distribution (\ref{eq:cnhv}) with parameters $(z=x+y,\nu,m,n)$ for all $x+y$, then both $X$ and $Y$ have the CMNB distributions with the parameters $(m,\nu ,p)$ and $(n,\nu ,p)$ respectively.
\end{theorem}
\begin{proof}
Note that $c(x,x + y) ={\left( {{{x+y} \choose {x}}\frac{{{\rm{B}}({m} + x,{n} +x+y-x)}}{{{\rm{B}}({m},{n})}}} \right)^\nu }/{{N(m,n,z,\nu )}} $. Then
$$c(a,b) ={\left( {{{b} \choose {a}}\frac{{{\rm{B}}({m+a} ,{n} + b-a)}}{{{\rm{B}}({m},{n})}}} \right)^\nu }/{{N(m,n,b,\nu )}}, $$so
 \begin{align*}
\frac{{c(x + y,x + y)c(0,y)}}{{c(x,x + y)c(y,y)}} &= {\left( {{\rm{B}}(m + x + y,n) \cdot {\rm{B}}(m,n + y)/{{x+y} \choose {x}}{\rm{B}}(m + x,n + y){\rm{B}}(m + y,n)} \right)^\nu }\\
 &= {\left( {\frac{{\Gamma (m + x + y)}}{{(x + y)!\Gamma (m)}}/\frac{{\Gamma (m + x)}}{{(x)!\Gamma (m)}}\frac{{\Gamma (m + y)}}{{(y)!\Gamma (m)}}} \right)^\nu }.
 \end{align*}
 We have $h(x) = {\left( {\frac{{\Gamma (m + x)}}{{(x)!\Gamma (m)}}} \right)^v},h(y) = {\left( {\frac{{\Gamma (m + y)}}{{(y)!\Gamma (m)}}} \right)^v}$. Apply (\ref{eq:fg}), $\frac{{h(y)c(0,y)}}{{c(y,y)}} = {\left( {\frac{{\Gamma (n + y)}}{{\Gamma (n)y!}}} \right)^\nu }$, let $p = {e^a}$ and compare with expression (\ref{eq:cnbd}), hence the proof is complete.
\end{proof}

\cite{rao64} study the following characterization of the Poisson
distribution: If $X$ is a discrete r.v. taking only nonnegative integer values and the
conditional distribution of $Y$ given $X = x$ is binomial distribution with parameters $\texttt{B}i(x, p) $ ($p$
does not depend on $x$), then $X$ follows the Poisson distribution iff
\[\mathrm{P}[Y = k] = P[Y = k|Y = X].\]

Based on the COM-negative hypergeometric distribution above and an extension of \cite{rao64}'s characterization which established by \cite{shanbhag77}, the Rao-Rubin characterization for CMNB is obtained.

\begin{lemma}\label{lem:shan}(\cite{shanbhag77})
Let $X,Y$ be the non-negative r.vs such that $\mathrm{P}(X = z) = {p_z}$ with\\ ${p_0} < 1,\,{p_z} > 0$, and
 \begin{align*}
\mathrm{P}(\left. {Y = k} \right|X = z) = \frac{{{a_r}{b_{z - k}}}}{{\sum\nolimits_{s = 0}^z  {{a_s}{b_{z - s}}} }} = \frac{{{a_k}{b_{z - k}}}}{{{c_z}}},(k = 0,1, \cdots ,z),
 \end{align*}
 where ${a_z} > 0$ for all $z \ge 0$, ${b_0},{b_1} > 0$ and ${b_z} \ge 0$ for $z \ge 2$, then
$$\mathrm{P}(Y = k) = \mathrm{P}(\left. {Y = k} \right|X = Y),(r = 0,1, \cdots ) \quad {\rm{iff}}\quad \frac{{{p_z}}}{{{c_z}}} = \frac{{{p_0}}}{{{c_0}}}{\theta ^z},(z = 0,1, \cdots )  \quad {\rm{for \ some}}\quad \theta  > 0 .$$
\end{lemma}

Next, we give the following CMNB extension of Rao-Rubin characterization. The Shanbhag's result is vital to prove this extension immediately.

\begin{theorem}
Let $X,Y$ be the non-negative integer-valued r.vs such that $\mathrm{P}(X = z) = {P_z}$ with ${P_0} < 1,{P_z} > 0$, and
 \begin{align*}
\mathrm{P}(\left. {Y = k} \right|X = z) = {\left( {{{z} \choose {k}}\frac{{{\rm{B}}({m} + k,{n} + z - k)}}{{{\rm{B}}({m},{n})}}} \right)^\nu }/N(m,n,z,\nu ),
 \end{align*}
then $\mathrm{P}(Y = k) = \mathrm{P}(\left. {Y = k} \right|X = Y),(r = 0,1, \cdots ) $ iff $X$ follows the CMNB distribution with the parameters $(m+n,\nu ,\theta)$ for some $\theta  > 0$.
\end{theorem}
\begin{proof}
Since the alternative form of negative hypergeometric distribution can be written as
\[{{z} \choose {k}}\frac{{{\rm{B}}(m + k,n + z - k)}}{{{\rm{B}}(m,n)}} = \frac{{\Gamma (m + k)}}{{{\rm{B}}(m,n)k!}}\frac{{\Gamma (n + z - k)}}{{{\rm{B}}(m,n)(z - k)!}}/\frac{{\Gamma (m + n + z)}}{{{\rm{B}}(m,n)z!}}.\]
From the normalizer in (\ref{eq:cnhv}), let $L(m,n,z,\nu ) = \frac{{N(m,n,z,\nu )}}{{{{[B(m,n)]}^{ - \nu }}}}$, we get
\[\mathrm{P}(\left. {Y = k} \right|X = z) = \frac{{{{[\Gamma (m + k)]}^\nu}}}{{{{(k!)}^\nu}\sqrt {L(m,n,z,\nu )} }} \cdot \frac{{{{[\Gamma (n + z - k)]}^\nu}}}{{{{((z - k)!)}^\nu}\sqrt {L(m,n,z,\nu )} }}/\frac{{{{[\Gamma (m + n + z)]}^\nu}}}{{{{(z!)}^\nu}}} = :\frac{{{a_k}{b_{z - k}}}}{{{c_z}}},\]
where ${a_z},{b_z},{c_z}$ satisfy the conditions in Lemma \ref{lem:shan}.

Then, ${c_z} = \frac{{{{[\Gamma (m + n + z)]}^\nu}}}{{{{(z!)}^\nu}}}$, compareing with form (\ref{eq:cnbd}), we conclude that $X$ is the CMNB distributed with the parameters $(m+n,\nu ,\theta)$.
\end{proof}

\subsection{COM-negative hypergeometric approximate to CMNB}

We have shown that the conditional distribution of CMNB r.v. $X$ given the sum of two CMNB random variables $X$ and $Y$ follows the COM-negative hypergeometric distribution (\ref{eq:cnhv}). Next, we show that the COM-negative hypergeometric distribution (\ref{eq:cnhv}) converges to the CMNB distribution (\ref{eq:cnbd}).

\begin{theorem}
Let $X$ be the COM-negative hypergeometric r.v. with parameters $(z,\nu,m,n)$ and p.m.f. $\mathrm{P}(X=k|z,\nu,m,n)$ given by (\ref{eq:cnhv}),
Assume $m<\infty$, $\frac{z}{{n + z}} = {p^{\frac{1}{\nu }}}$, $p \in (0,1)$, then
 \begin{align*}
\lim_{n\rightarrow\infty} \mathrm{P}(\left. {X = k} \right|z,\nu,m,n) = {\left(\frac {\Gamma(m+k)}{k!\Gamma(m)}  \right)^\nu }p^k(1-p)^{m}\frac {1}{C(m,\nu,p)}.
 \end{align*}
\end{theorem}
\begin{proof}
First, we multiply the p.m.f of COM-negative hypergeometric distribution by the normalization constant (see (\ref{eq:cnhv})):
 \begin{align*}
& N(m,n,z,\nu )\mathrm{P}(X = k|z,\nu ,m,n)
= {\left( {\frac{{z!}}{{k!(z - k)!}} \cdot \frac{{\Gamma (m + k)\Gamma (n + z - k)}}{{\Gamma (m + n + z)}} \cdot \frac{{\Gamma (m + n)}}{{\Gamma (m)\Gamma (n)}}} \right)^\nu }\\
 &= {\left( {\frac{\Gamma (m + k)}{{k!\Gamma (m )}}}\right)^\nu
 \left( {\frac{{z(z-1)\cdots (z-k+1)\Gamma (n+z - k)\Gamma (n + m)}}
 {{(m+n+z-1)(m+n+z-2)\cdots (m+n+z-k)\Gamma (m+n+z-k)\Gamma (n)}}} \right)^\nu}.
 \end{align*}
Using the Stirling's formula for the Gamma function, $\Gamma (z) = \sqrt {\frac{{2\pi }}{z}} {\mkern 1mu} {\left( {\frac{z}{e}} \right)^z}\left( {1 + O\left( {\frac{1}{z}} \right)} \right)$, it yields

 \begin{align*}
& \frac{{\Gamma (n + z - k)}}{{\Gamma (m + n + z - k)}} \cdot \frac{{\Gamma (m + n)}}{{\Gamma (n)}}\\
& = \frac{{\sqrt {\frac{{2\pi }}{{n + z - k}}} {{\left( {\frac{{n + z - k}}{e}} \right)}^{n + z - k}}\sqrt {\frac{{2\pi }}{{n + m}}} {{\left( {\frac{{n + m}}{e}} \right)}^{n + m}}\left( {1 + O\left( {\frac{1}{{n + z - k}}} \right)} \right)\left( {1 + O\left( {\frac{1}{{n + m}}} \right)} \right)}}{{\sqrt {\frac{{2\pi }}{n}} {{\left( {\frac{n}{e}} \right)}^n}\sqrt {\frac{{2\pi }}{{n + m + z - k}}} {{\left( {\frac{{n + m + z - k}}{e}} \right)}^{n + m + z - k}}\left( {1 + O\left( {\frac{1}{{n + m + z - k}}} \right)} \right)\left( {1 + O\left( {\frac{1}{n}} \right)} \right)}}\\
&=\sqrt {\frac{{n(n + m + z - k)}}{{(n + z - k)(m + n)}}} {\left( {\frac{{n + m}}{{m + n + z - k}}} \right)^m}\frac{{{{\left( {1 + \frac{m}{n}} \right)}^n}}}{{{{\left( {1 + \frac{m}{{n + z - k}}} \right)}^{n + z - k}}}}\frac{{\left( {1 + O\left( {\frac{1}{{n + z - k}}} \right) + O\left( {\frac{1}{{n + m}}} \right)} \right)}}{{\left( {1 + O\left( {\frac{1}{{n + m + z - k}}} \right) + O\left( {\frac{1}{n}} \right)} \right)}}.
\end{align*}
From the conditions, we have $\frac {m+n}{m+n+z-k}\rightarrow 1- p^{\frac{1}{\nu}}$ as $n\rightarrow\infty$. Then
\[\mathop {\lim }\limits_{n \to \infty } {\left( {\frac{{\Gamma (n + z - k)}}{{\Gamma (m + n + z - k)}} \cdot \frac{{\Gamma (m + n)}}{{\Gamma (n)}}} \right)^\nu } = {\left( {1 - {p^{1/\nu }}} \right)^{m\nu }} = {\left( {1 - p} \right)^m}{\left( {\frac{{{{(1 - {p^{1/\nu }})}^v}}}{{1 - p}}} \right)^m},\]
and
\[\mathop {\lim }\limits_{n \to \infty } {\left( {\frac{{z(z - 1)\cdots (z - k + 1)}}{{(m + n + z - 1)(m + n + z - 2)\cdots (m + n + z - k)}}} \right)^\nu } = {p^k}.\]

Finally, we obtain
\[\mathop {\lim }\limits_{n \to \infty } \mathrm{P}(X = k|z,\nu ,m,n) = {\left( {\frac{{\Gamma (m + k)}}{{k!\Gamma (m)}}} \right)^\nu }{p^k}{(1 - p)^m} \cdot \mathop {\lim }\limits_{n \to \infty } {[N(m,n,z,\nu )]^{ - 1}}{\left( {\frac{{{{(1 - {p^{1/\nu }})}^\nu }}}{{1 - p}}} \right)^m}\]
Note that the normalizing constant $ \mathop {\lim }\limits_{n \to \infty } {[N(m,n,z,\nu )]^{ - 1}}{\left( {\frac{{{{(1 - {p^{1/\nu }})}^\nu }}}{{1 - p}}} \right)^m }$ exsists and does not depend on $k$. Therefore, $\mathop {\lim }\limits_{n \to \infty } \mathrm{P}(X = k|z,\nu ,m,n)$
is the p.m.f. of ${\rm{CMNB}}(m,\nu,p)$ via comparing with (\ref{eq:cnbd}).
\end{proof}

\cite{borges14} showed that the COM-Poisson distribution is the limiting distribution of the COM-binomial distribution. Let $X_{m} \sim {\rm{CMB}}(m,p,\nu)$(see (\ref{eq:cb})) and $\lambda = m^{\nu} p $ for $\nu \geq 0$, then
$$
\lim _{m \rightarrow \infty} \frac{ {{{{{m} \choose {k}}}^\nu}{p^k}{{(1 - p)}^{m - k}}} }{N(m,p,\nu)}
= \frac{\lambda ^{k}}{(k!)^{\nu}} \cdot Z^{-1} (\lambda, \nu)
$$
namely $ \lim _{m \rightarrow \infty} X _{m} \sim {\rm{CMP}}(\lambda, \nu)$.

To sum up, judging from the above mentioned theorems of limiting distribution, we may naturally draw the relationships among some COM-type distributions in Figure \ref{relation}.
\begin{figure}[H]
\centering
\includegraphics[width=0.75\textwidth]{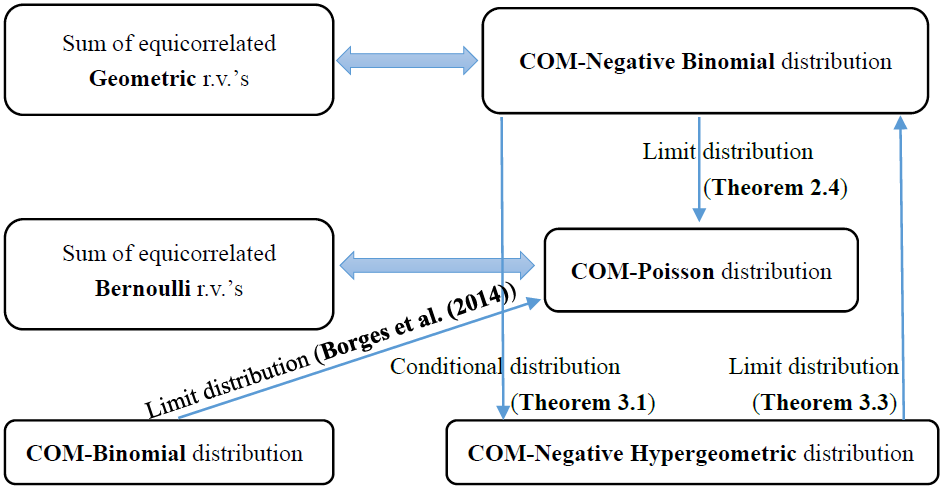}
\caption{Relationships among some COM-type distributions} \label{relation}
\end{figure}

\subsection{Functional operator characterization by Stein's identity}
Using recursive formula (\ref{eq:rf}) for CMNB distribution, an extension of functional operator characterization which arises from the negative binomial approximation literatures by Stein-Chen method is presented. The following Stein's lemma also can be derived from the work of \cite{brown01} who give a large class of approximation distribution which is the equilibrium distribution of a birth-death process with arbitrary arrival rate and service rate.

The functional operator characterization is well-known in
the Stein-Chen method literatures. Lemma \ref{lem:foc} extends the Lemma 1 in \cite{brown99} for negative binomial approximation.

\begin{lemma}\label{lem:foc}
The r.v. $W$ has distribution ${\rm{CMNB}}(r,\nu ,p)$iff the equation
\begin{equation}
\E[{W^\nu}g(W) - p{(W + r)^\nu }g(W + 1)] = 0
\end{equation}
 holds for any bounded function $g: \Bbb{Z}^{+} \to \Bbb{R}$.
\end{lemma}

\begin{proof}
  Sufficiency: When $W$ is CMNB distributed in the form of
  \eqref{eq:cnbd}, $f$ is a bounded function. Then we have
  \begin{align*}
\E[{W^\nu}g(W)] &= \sum\limits_{k = 0}^\infty  {{k^v}} g(k){\left( {\frac{{\Gamma (k + r)}}{{k!\Gamma (r)}}} \right)^\nu }\frac{{{p^k}{{(1 - p)}^r}}}{{C(r,\nu ,p)}}= \sum\limits_{k = 1}^\infty  {p{{\left( {k - 1 + r} \right)}^\nu }g(k)} {\left( {\frac{{\Gamma (k - 1 + r)}}{{(k - 1)!\Gamma (r)}}} \right)^\nu }\frac{{{p^{k - 1}}{{(1 - p)}^r}}}{{C(r,\nu ,p)}}\\
 &= \E[p{(W + r)^\nu }g(W + 1)].
\end{align*}
Hence $\E g(W) = 0$.

Necessity: $\E g(W) = 0$ for all bounded functions. Letting $f(W) = {1_{(W = k)}}(W)$, a simple computation
shows the recursive formula (\ref{eq:rf}). This verifies that  $W \sim {\rm{CMNB}}(r,\nu ,p)$.
\end{proof}

\section{Generating CMNB-distributed random variables}
\label{generation}
The inverse transform sampling is a basic method in pseudo-random number sampling, namely for generating sample numbers at random from any probability distribution given its cumulative distribution function. The following lemmas are the basis of inverse method.
\begin{lemma}\label{lemma1}
  Suppose $F(x)$ is the cumulative distribution function of random variable $\xi$, then $\theta = F(\xi)$ is uniformly distributed in interval $[0,1]$.
\end{lemma}
\begin{lemma}\label{lemma2}
  Suppose $\theta\sim U[0,1]$, for any cumulative distribution function $F(x)$, let $\xi=F^{-1}(\theta)$, then $\xi\sim F(x)$, that is, $F(x)$ is the distribution function of random variable $\xi$.
\end{lemma}
The proof of Lemma \ref{lemma1} and \ref{lemma2} is trivial and can be found in text books of computational statistics.

Consider generating random sample from CMNB$(r,\nu,p)$, since the p.m.f. is
\begin{equation}
{\rm{P}}(X = k) =
\frac {{{{(\frac{{\Gamma (r + k)}}{{k!{\mkern 1mu} \Gamma (r)}})}^\nu }{p^k}{{(1 - p)}^r}}}
      {{\sum\limits_{i = 0}^\infty  {{{(\frac{{\Gamma (r + i)}}{{i!{\mkern 1mu} \Gamma (r)}})}^\nu }} {p^i}{{(1 - p)}^r}}}
={{\left(\frac{{\Gamma (r + k)}}{{k!{\mkern 1mu} \Gamma (r)}}\right)}^\nu }
{{p^k}{{(1 - p)}^r}}
\frac{1}{{C(r,\nu ,p)}},\quad (k = 0,1,2, \ldots ).
\end{equation}
It follows easily that
\begin{equation}\label{recursion}
\left\{
  \begin{aligned}
    \text{P}(X=0) &= (1 - p)^r\frac{1}{C(r,\nu,p)}\\
    \frac{\text{P}(X=k)}{\text{P}(X=k-1)} &= p\cdot (\frac{k+r-1}{k})^{\nu},\quad (k = 1,2, \ldots ).
  \end{aligned}
\right.
\end{equation}
Multiplying these equations yields
\begin{align*}
  \text{P}(X=k) &= p\cdot (\frac{k+r-1}{k})^{\nu}\cdot \text{P}(X=k-1) \\
                &= p^2\cdot (\frac{k+r-1}{k}\frac{k+r-2}{k-1})^{\nu}\cdot \text{P}(X=k-2) \\
                &= p^k\cdot (\frac{k+r-1}{k}\frac{k+r-2}{k-1}\cdots \frac{r}{1})^{\nu}\cdot \text{P}(X=0).
\end{align*}
Thus, for any non-negative $k$, the distribution function $F(k) = \sum_{i=0}^{k}\text{P}(X=i)$ can be easily calculated.
By Lemma \ref{lemma2}, the inverse method alternates between the following steps:
{\small
\begin{description}
  \item[Step 1:] Let $j=0$, $\text{P}(X=j)=(1 - p)^r/C(r,\nu,p)$;
  \item[Step 2:] Generate a random sample $\theta\sim U[0,1]$, and fix its value;
  \item[Step 3:] Let $F(j+1) = F(j) + \text{P}(X=j+1)$;
  \item[Step 4:] If $F(j)<\theta\leq F(j+1)$, then $x=j+1$; Otherwise, let $j=j+1$, and back to Step 3;
\end{description}}

We conclude that, $x$ is the random number drawn from the CMNB distribution with parameters $(r,\nu,p)$.

\section{Estimations}

\subsection{Estimation by three recursive formulas}\label{subsection:recursive}
The method of recursive formula estimation (or refer as \emph{ratio regression}, see \cite{Bohning16}) which give crude estimations of the estimated parameters, is originated with \cite{shmueli05}, and extended by \cite{imoto14}. However, this crude estimations can be put into the maximum likelihood estimation (MLE) with Newton-Raphson algorithm as initial values.

First, we note that the p.m.f. of CMNB has the following recursive formula:
\begin{equation}\label{eq:estimation1}
  \frac{P_{k+1}}{P_{k}}=p\cdot(\frac{k+r}{k+1})^\nu,\quad(k=0,1,\ldots).
\end{equation}

Second, applied with the expression of log-concave~(log-convex), we have
\begin{equation}\label{eq:estimation2}
  \log(\frac{P_{k}P_{k+2}}{P_{k+1}^2})=\nu\log(\frac{k+r+1}{k+2}\cdot\frac{k+1}{r+k}).
\end{equation}

Replace ``$k$" by ``$k+1$" in \eqref{eq:estimation2}. Then we obtain
\begin{equation}\label{eq:estimation3}
  \frac{\log(\frac{P_{k}P_{k+2}}{P_{k+1}^{2}})}{\log(\frac{P_{k+1}P_{k+3}}{P_{k+3}^{2}})}=\frac{\log(\frac{k+r+1}{k+2}\cdot\frac{k+1}{r+k})}{\log(\frac{k+r+2}{k+3}\cdot\frac{k+2}{r+k+1})}.
\end{equation}

Thus, $(\nu,r,p)$ could be sequentially solved by the system of equations \eqref{eq:estimation1}, \eqref{eq:estimation2} and \eqref{eq:estimation3}, if we have the sample p.m.f.

\subsection{Maximum likelihood estimation}
\label{sec:mle}
Let r.v. $X$ be distributed as the CMNB distribution with parameters $\theta=(r,\,\nu,\,p)$. We consider the MLE in the case parameters $r$, $\nu$, $p$ are unknown. The likelihood function is
\begin{equation}
  \begin{aligned}
  L(\theta) &= \prod_{i=1}^{n}\mathrm{P}(X=x_i) = \prod_{i=1}^{n} (\frac{\Gamma(r+x_i)}{x_i!})^\nu p^{x_i} \Gamma(r)^{-\nu} (1-p)^r C(r,\nu,p)^{-1}\\
               &= [\prod_{i=1}^{n}{\frac{\Gamma(r+x_i)}{x_i!}}]^\nu p^{\sum_{i=1}^{n}x_i} \Gamma(r)^{-n\nu} (1-p)^{nr} C(r,\nu,p)^{-n},
\end{aligned}
\end{equation}
where $n$ is the sample size, $x_1,x_2,\ldots,x_n$ are the observed values, and the log-likelihood function is given by
\begin{align}
  \log L(\theta) &= \nu\sum_{i=1}^{n}\log\frac{\Gamma(r+x_i)}{x_i!} + \log(p)\sum_{i=1}^{n}x_i  + nr\log(1-p) - n\nu\log(\Gamma(r)) - n\log(C(r,\nu,p))
\end{align}
To find the maximum point, for $\log L(\theta)$, we take the partial derivatives with respect to $r$, $\nu$ and $p$ and set them equal to zero, hence the likelihood equation is given by
\begin{equation}\label{mle}
\left\{
  \begin{aligned}
    F_1(\theta) &= \frac{\partial\log L(\theta)}{\partial r} = \nu\sum_{i=1}^{n}\psi(r+x_i) + n\log(1-p) - n\nu\psi(r) - n\frac{\partial \log(C(r,\nu,p))}{\partial r}=0;\\
    F_2(\theta) &=\frac{\partial\log L(\theta)}{\partial \nu} = \sum_{i=1}^{n}\log\frac{\Gamma(r+x_i)}{x_i!} - n\log(\Gamma(r)) - n\frac{\partial \log(C(r,\nu,p))}{\partial \nu}=0;\\
    F_3(\theta) &=\frac{\partial\log L(\theta)}{\partial p} = \frac{\sum_{i=1}^{n}x_i}{p} + \frac{nr}{p-1} - n\frac{\partial \log(C(r,\nu,p))}{\partial p}=0;
  \end{aligned}
\right.
\end{equation}
where $\psi(\alpha)=\Gamma^{'}(\alpha)/\Gamma(\alpha)$ is called the digamma function.
The analytical solutions of the above likelihood equations are not tractable, therefore, numerical optimization method is used to obtain the maximum likelihood estimates.
Let $F(\theta)=\big(F_1(\theta),\,F_2(\theta),\,F_3(\theta)\big)^T,$
 then the Fisher information matrix $I(r,\nu,p)$ would be the Jacobin matrix of $F(\theta)$. Given trial values $\theta_k=(r_k,\,\nu_k,\,p_k)$, applied with the scoring method for solving \eqref{mle},  we can update to $\theta_{k+1}=(r_{k+1},\,\nu_{k+1},\,p_{k+1})$ as
\begin{equation}
\left( \begin{array}{l}
{r_{k + 1}}\\
{\nu _{k + 1}}\\
{p_{k + 1}}
\end{array} \right) = \left( \begin{array}{l}
{r_k}\\
{\nu _k}\\
{p_k}
\end{array} \right) + {[I(r,\nu ,p)]^{ - 1}}F({r_k},{\nu _k},{p_k}).
\end{equation}
For the initial point $\theta_0=(r_0,\, \nu_0,\, p_0)$, we can choose the crude estimated parameters introduced in Subsection \ref{subsection:recursive}.

\subsection{Kolmogorov-Smirnov and Chi-squared goodness-of-fit test}
In goodness-of-fit test of discrete distributions, Pearson's chi-squared test is a popular choice to check distribution model, which is usually better than the others. The larger p-value, the better goodness-of-fit would arise from the assumed models. However, when the data come from an assumed model, with $r$ unknown parameters $\theta=({\theta _1}, \cdots ,{\theta _r})$ and the null hypothesis ${H_0}$, what we are interested in is
\[{H_0}:F(x) = {F_0}(x;{\theta _1}, \cdots ,{\theta _r}),\]
 The $\chi^2$ statistic has the following limiting distribution:
\[\eta  = \sum\limits_{i = 1}^m {\frac{{{{({n_i} - n{{\hat p}_{0i}})}^2}}}{{n{{\hat p}_{0i}}}}} \mathop  \to \limits^d {\chi ^2}(m - 1 - r),\]
where $n_i$ is the number of class $i$, and $m$ classes ${A_i}~(i = 1,2, \cdots ,m)$ which are the sub-division of $n$ samples, it satisfies that:
\[\sum\nolimits_{i = 1}^m {{n_i}}  = n,( - \infty ,\infty ) =  \cup _{i = 1}^m{A_i},{\hat p_{i0}} = :\frac{{\# \{ x \in {A_i})\} }}{n}.\]
We also have the another handy expression for calculating $\chi^2$ statistic: $\eta  = \sum\limits_{i = 1}^m {\frac{{n_i^2}}{{n{{\hat p}_{0i}}}} - n} $.

One limitation of $\chi^2$ statistic is that, if degree of freedom are too small, the approximation to the $\chi^2$ distribution would fail. For example, $k=m - 1 - r = 0,1,2$. Since $\chi^2$ test of small degree of freedom~(denotes it by $k$) did not have enough power (namely ${\mathrm{P}_{{\theta}}}(\eta  > {\eta _0})$ is not large enough). That is to say, for a $\chi^2$ r.v. $\eta $, the p-value is defined by
\[p({\eta _0}) = {\mathrm{P}_{{\theta}}}(\eta  > {\eta _0}) = \int_{{\eta _0}}^{ + \infty } {\frac{1}{{{2^{{\textstyle{k \over 2}}}}\Gamma ({\textstyle{k \over 2}})}}{x^{k - 1}}{e^{ - \frac{x}{2}}}} dx,\]
which tends to be small when $k$ varies from 3 to 1. An attempt at this has been made in the figure below.

\begin{figure}[H]
\centering
\includegraphics[height=0.5\textwidth=0.7]{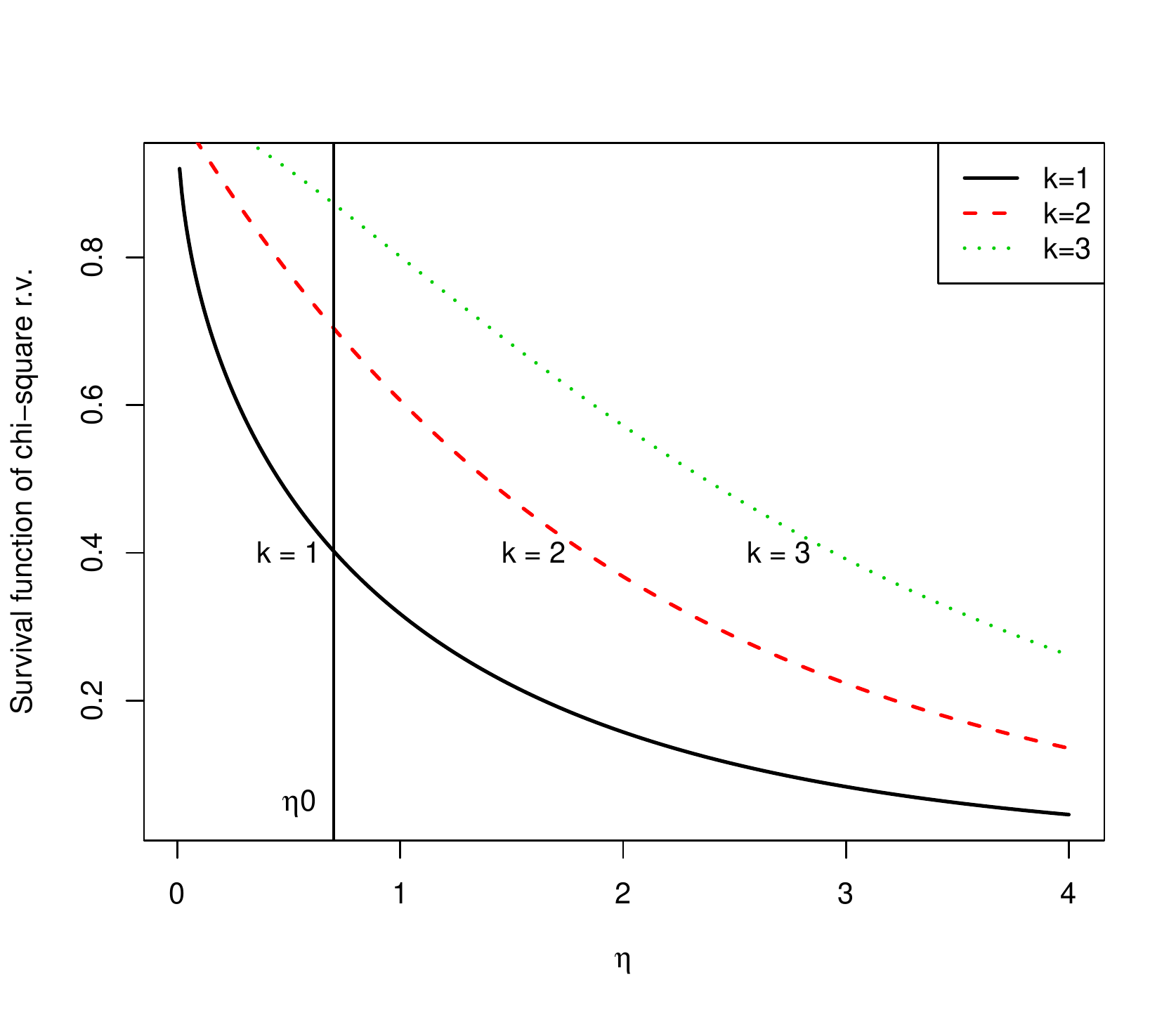}
\caption{ The survival function of $\chi^2$ r.v. as $k=1,2,3$}
\end{figure}

Another drawback is that, with small number of class $m$, \cite{haberman88} discussed that: ``chi-squared test statistics may be asymptotically
inconsistent even in cases in which a satisfactory chi-squared approximation exists for the distribution under the null hypothesis."

Consequently, a better way of comparing distributions is to use a non-parametric test not depending on the parametric assumption of the models. One of notable non-parametric test is Kolmogorov-Smirnov test by using following statistic:
\[{D_n} = \mathop {\sup }\limits_x \left| {{{\hat F}_n}(x) - {F_0}(x)}\right|,\]
which is a nonparametric test of continuous distributions. And it could be modified for discrete distribution:
\[{D_n} = \mathop {\sup }\limits_x \left| {{{\hat F}_n}(x) - {F_0}(x)} \right| = \mathop {\max }\limits_{i \in I} [\max (|{{\hat F}_n}({x_i}) - {F_0}({x_i})|,\mathop {\lim }\limits_{x \to {x_i}} |{{\hat F}_n}({x_i}) - {F_0}({x_{i - 1}})|)],\]
where $x_i$'s are the discontinuity points ($I$ is the countable index set).

Fortunately, \cite{arnold11} developed the R package ``dgof", which can calculate Kolmogorov-Smirnov test of the discrete distribution. The non-parametric test avoids the assumption of parametric model, hence it is an effective way to evaluate different distributions' performance.

\section{Applications with simulated and real data}

In this section, we will describe four examples of fitting data by CMNB distribution, and we will compare them with those by the negative binomial and COM-Poisson distribution. The notations, CMNB, NB and CMP distribution, in the following tables, are representing the CMNB, negative binomial and COM-Poisson, respectively. By using scoring method in section \ref{sec:mle}, we could estimated the CMNB distribution with three parameters. The original and expected frequencies, parameter estimators (obtained by maximum likelihood method), the $\chi^2$ and  K-S statistics, and corresponding $p$-values. are all being considered in the tables below.

\subsection{Simulated Data}
In this subsection, the inverse method mentioned in section \ref{generation} were used to generate random variables from CMNB, NB and COM-Poisson distribution with certain parameters. And we compare the performance of the fit by CMNB, NB and COM-Poisson distribution.

\textbf{Example 1.}
As our first simulation example, we consider the original data that are following the NB distribution. We use inverse method to draw 10000 random samples from NB distribution with parameter $(r=1, p=0.5)$, and fit the data with above-mentioned distributions, the results are shown in the following table \ref{simulate_nb}.
\begin{table}[H]
\footnotesize
\caption{Fitting results of simulated data from NB$(1, 0.5)$}\label{simulate_nb}
\centering
\begin{tabular}{lllll}
\hline
Sample values & Frequency & \multicolumn{3}{c}{Fitted Values} \\
\cline{3-5}
\multicolumn{2}{l}{} & CMNB & NB & CMP \\
\hline
  0 & 5060 & 5054 & 5049 & 5019 \\
  1 & 2480 & 2494 & 2492 & 2502 \\
  2 & 1199 & 1237 & 1236 & 1248 \\
  3 & 638 & 615 & 614 & 622 \\
  4 & 318 & 306 & 306 & 310 \\
  5 & 165 & 152 & 152 & 155 \\
  6 & 74 & 76 & 76 & 77 \\
  7 & 33 & 38 & 38 & 38 \\
  8 & 20 & 19 & 19 & 19 \\
  9 & 8 & 9 & 9 & 10 \\
  10 & 4 & 5 & 5 & 5 \\
  11 & 1 & 2 & 2 & 2 \\
Total & 10000 &10007  &9998 &10007\\
\hline
  par1 & &0.97 & 0.99 & 0.50 \\
  par2 & &1.00 & 0.50 & 0.00 \\
  par3 & &0.51 &  &  \\
  $\chi^2$ & &5.29 & 5.44 & 5.83 \\
  d.f. of $\chi^2$ & &8.00 & 9.00 & 9.00 \\
  p value of $\chi^2$ & &0.73 & 0.79 & 0.76 \\
  K-S & &0.002199 &0.003976 &0.004451 \\
  p value of K-S & &1.000000 &0.997435 &0.988831 \\
   \hline
\end{tabular}
\end{table}
In this case, although the data are generated from NB distribution, the CMNB distribution can also be recognized as the true distribution, as the estimator of $\nu$ is 1. That is to say, the data can be seen negative binomial distributed, and note that he result of  $\chi^2$ test and K-S test is nearly the same, all of these indicate that the CMNB distribution is an extension of NB distribution.

\textbf{Example 2.}
In this example, we will evaluate the performance of NB distribution, when the original data are generated from CMNB distribution. Assume $X_1, X_2, \ldots, X_n$ are independent and identically CMNB distributed random variables with parameters $(r,\nu,p)$. Let $n = 10000$ and $(r,\nu,p) = (0.005, 0.1, 0.5)$, we can generate $n$ sample points $x_1, x_2,\ldots,x_n$ through inverse method introduced in section \ref{generation}. The fitting results are summarized below:
\begin{table}[H]
\footnotesize
\caption{Fitting results of simulated data from CMNB(0.005, 0.1 0.5)}\label{simulated_cmnb}
\centering
\begin{tabular}{lllll}
\hline
Sample values & Frequency & \multicolumn{3}{c}{Fitted Values} \\
\cline{3-5}
\multicolumn{2}{l}{} &  CMNB & NB & CMP \\
\hline
  0 & 6442 & 6449 & 6421 & 5937 \\
  1 & 1866 & 1878 & 1951 & 2413 \\
  2 & 874 & 869 & 837 & 981\\
  3 & 435 & 415 & 394 & 399\\
  4 & 188 & 201 & 194 & 162\\
  5 & 101 & 98 & 98 & 66 \\
  6 & 55 & 48 & 50 & 27 \\
  7 & 19 & 24 & 26 & 11\\
  8 & 7 & 12 & 14 & 4 \\
  9 & 8 & 6 & 7 & 2 \\
   10 & 2 & 3 & 4 & 1\\
   11 & 2 & 1 & 2 & 0 \\
   12 & 1 & 1 & 1 & 0\\
   Total & 10000 & 10005  & 9999 & 10003 \\
\hline
  par1 & & 0.01 & 0.55 & 0.41 \\
  par2 & & 0.12 & 0.45 & 0.00 \\
  par3 & & 0.50 &  &  0.44\\
  $\chi^2$ & &7.26 & 16.38 & 281.94 \\
  d.f. of $\chi^2$ & &9.00 & 10.00 & 10.00 \\
  p value of $\chi^2$ & &0.61 & 0.09 & 0.00 \\
  K-S & &0.001484 & 0.006484 & 0.050678 \\
  p value of K-S & &1.000000 & 0.794579 & 0.000000 \\
   \hline
\end{tabular}
\end{table}
From table \ref{simulated_cmnb}, we note first that in $\chi^2$ test, the CMNB distribution's $\chi^2$ statistic 7.26 is dramatically smaller than NB distribution's 16.38, while the corresponding p-values are 0.61 and 0.09 respectively, which implies that NB distribution may not be a reasonable fit. As for the result of the K-S test, it suggests that CMNB distribution performs best, because the p-value of CMNB distribution is approximately 1, while NB's and COM-Poisson's are 0.79 and 0.

\subsection{Real data analysis}
 In this subsection, the distributions mentioned above are considered here to analyze real actuarial claim data that have ultrahigh zero-inflated and overdispersion properties, then compare its Kolmogorov-Smirnov test and Chi-squared test.

\textbf{Example 3.}
Let us consider the claim counts of the third party liability vehicle insurance (see \cite{willmot87} for data set in an Zaire insurance company) which correspond to claims from 4000 vehicle policies. \cite{G¨®mez-D¨¦niz2014} analyzed the data using negative binomial distribution and found that it is a reasonable fit. We analyze the data by CMNB, NB and CMP distribution, and the results are summarized below:
\begin{table}[H]
\footnotesize
\caption{Fit of Willmot2 data}\label{Willmot}
\centering
\begin{tabular}{lllll}
\hline
No. of claims & Frequency & \multicolumn{3}{c}{Fitted Values} \\
\cline{3-5}
\multicolumn{2}{l}{} &  CMNB & NB & CMP \\
\hline
  0 & 3719 & 3720 & 3719 & 3681 \\
  1 & 232 & 231 & 230 & 294 \\
  2 & 38 & 39 & 40 & 23 \\
  3 & 7 & 8 & 8 & 2 \\
  4 & 3 & 2 & 2 & 0 \\
  5 & 1 & 1 & 0 & 0 \\
Total &  4000 &4001 &3999 &3991\\
\hline
  par1 & &0.57 & 0.22 & 0.08 \\
  par2 & &3.06 & 0.71 & 0.00 \\
  par3 & &0.35 &  &  \\
  $\chi^2$ & &1.01 & 1.56 & 173.28 \\
  d.f. of $\chi^2$ & &2.00 & 3.00 & 3.00 \\
  p value of $\chi^2$ & &0.60 & 0.67 & 0.00 \\
  K-S & &0.000250 & 0.000500 & 0.009500 \\
  p value of K-S & &1.000000 & 1.000000 & 0.863178 \\
\hline
\end{tabular}
\end{table}
It's a well known fact that NB distribution is a popular choice to fit the claim data in actuarial science. However findings in Table \ref{Willmot} suggest that, CMNB might be a better choice for fitting this data. It appears that, for chi-square statistic, fitting the chi-square statistic with the NB distribution, it turns out to be $1.56$, which is dramatically larger than CMNB's $1.01$, and the p-values correspond to the $\chi^2$ statistics is nearly the same. As for Kolmogorov-Smirnov test, although the p-value is both 1, it still can be seen that the K-S statistic of CMNB is slightly smaller than NB's. All of these show that the CMNB is superior to the fit the data.

\textbf{Example 4.}
For this example, we use the car insurance claim data of a Chinese insurance company (see \cite{Wang2000}), which were modeled with negative binomial distribution. Total insurance policies are $n=35072$. We analyze the data with the above-mentioned distributions, and the results are shown in table \ref{real_2}.
\begin{table}[H]
\footnotesize
\caption{Fit of Car insurance claim data}\label{real_2}
\centering
\begin{tabular}{lllll}
\hline
No. of Claims & Frequency & \multicolumn{3}{c}{Fitted Values} \\
\cline{3-5}
\multicolumn{2}{l}{} &  CMNB & NB & CMP \\
\hline
  0 & 27141 & 27177 & 27166 & 26599 \\
  1 & 5789 & 5666 & 5664 & 6430 \\
  2 & 1443 & 1554 & 1563 & 1554 \\
  3 & 457 & 466 & 467 & 376 \\
  4 & 155 & 146 & 145 & 91 \\
  5 & 56 & 47 & 46 & 22 \\
  6 & 27 & 15 & 15 & 5 \\
  7 & 2 & 5 & 5 & 1 \\
  8 & 1 & 2 & 2 & 0 \\
  9 & 1 & 1 & 1 & 0 \\
Total & 35072 &35079 &35074 &35078 \\
\hline
  par1 & &0.95 & 0.61 & 0.24 \\
  par2 & &10.40 & 0.66 & 0.00 \\
  par3 & &0.36 &  &  \\
  $\chi^2$ & &24.16 & 27.88 & 300.60 \\
  d.f. of $\chi^2$ & &6.00 & 7.00 & 7.00 \\
  p value of $\chi^2$ & &0.00 & 0.00 & 0.00 \\
  K-S & &0.002667 & 0.002905 & 0.015584 \\
  p value of K-S & &0.964199 & 0.928729 & 0.000000 \\
   \hline
\end{tabular}
\end{table}
As can be observed from table \ref{real_2}, CMNB is the clear winner. CMNB distribution outperforms other distributions by either $\chi^2$ test or K-S test, as $\chi^2$ statistic and K-S statistic of CMNB is slightly smaller than that of NB's, and dramatically smaller than CMP's. It shows that their performance is on the whole dominated by CMNB distribution.

\section{Further researches}
Many discrete distributions in aspects of related statistical models are widely proposed in plenty of research articles. The CMNB distribution can be used in driving count data in some generalized linear models (GLMs). For example, Poisson regression, COM-Poisson regression and negative binomial regression (NBR), whose driving distribution of count data are actually the special case or limiting case of CMNB distribution, see Figure \ref{relation} for a visual relations. Even in the popular logistic regression, the assuming distribution lying in this model is Bernoulli distribution which is a limiting case of COM-Poisson distribution. Besides GLMs, high-dimensional GLMs (i.e., the sample size $n$ may be greater than the number of covariates $p$) should also be considered for CMNB, see \cite{zhangjia17} and references therein for high-dimensional NBR.
The discrete frailty item as a random effect in survival analysis model (or Long-term survival models, see \cite{rodrigues12}), which could be considered to set some flexible distributions. Besides the researches of regression model, in our paper, both conditional distribution and Stein identity characterizations are obtained. Test of statistic based on the novel characterization is more reasonable, since that the distribution-free and omnibus goodness-of-fit test like Kolmogorov-Smirnov and Chi-squared does not depend on certain types(or families) of distributions, may have their own drawbacks. For example, some new tests of a class of count distributions which includes the Poisson is constructed from Stein identity characterization~(see \cite{meintanis08}), and some tests of logarithmic series distribution comes from conditional distribution characterization~(see \cite{ramalingam84}). The goodness-of-fit tests between two continuous probability distributions based on combining Stein¡¯s identity is studied by \cite{Liu16}, and we consider study Stein discrepancy goodness-of-fit tests for discrete distribution in the future.

\section{Acknowledgments}
The proposed COM-negative binomial distribution of this work was
as early as conceptualized in Dec. 2014 when the authors saw the online version of \cite{imoto14}. The authors want to thank Prof. K{\"o}hler R. for mailing the valuable encyclopedia of discrete univariate distributions (\cite{wimmer99}) to us. This work was partly supported by the National Natural Science Foundation of China (No.11201165).

\end{document}